\documentclass{amsart}

\usepackage{color}
\usepackage{amsmath}
\usepackage{amsfonts}
\usepackage{amssymb,enumerate}
\usepackage{enumitem}
\usepackage{amsthm}
\usepackage{comment}
\usepackage[all]{xy}
\usepackage{hyperref}
\usepackage[utf8]{inputenc}
\usepackage[noadjust]{cite}

\theoremstyle{definition}
\newtheorem{lem}{Lemma}[section]
\newtheorem{cor}[lem]{Corollary}
\newtheorem{prop}[lem]{Proposition}
\newtheorem{thm}[lem]{Theorem}

\newtheorem{conj}[lem]{Conjecture}
\newtheorem{defn}[lem]{Definition}

\newtheorem{rem}[lem]{Remark}

\renewcommand{\geq}{\geqslant}
\renewcommand{\leq}{\leqslant}

\hypersetup{nolinks=true}
\numberwithin{equation}{section}

\begin{document}
\author{Jason Boynton}
\address{Department of Mathematics\\
	North Dakota State University\\
	Fargo, ND 58103}
\email[J.~Boynton]{jason.boynton@ndsu.edu}
\keywords{Half-factorial, HFD, orders, factorization}
\subjclass[2010]{Primary: 13F15, 13A05, 13B22}
\author{Jim Coykendall}
\address{School of Mathematical and Statistical Sciences\\
	Clemson University\\
	Clemson, SC 29634}
\email[J.~Coykendall]{jcoyken@clemson.edu}
\author{Grant Moles}
\address{Mathematics Department\\
	Union College\\
	Schenectady, NY 12308}
\email[G.~Moles]{molesg@union.edu}

\author{Chelsey Morrow}
\address{Mathematics Department\\
	Culver-Stockton College\\
	Canton, MO 63435}
\email[C.~Morrow]{cmorrow@culver.edu}

\title{Overrings of Half-Factorial Orders}
\begin{abstract}
The behavior of factorization properties in various ring extensions is a central theme in commutative algebra. Classically, the UFDs are (completely) integrally closed and tend to behave well in standard ring extensions, with the notable exception of power series extension. The half-factorial property is not as robust; HFDs need not be integrally closed and the half-factorial property is not necessarily preserved in integral extensions or even localizations. Here we exhibit classes of HFDs that behave well in (almost) integral extensions, resolve an open question on the behavior of the boundary map, and give a squeeze theorem for elasticity in certain domains.
\end{abstract}

\maketitle
\sloppy

\section{Introduction}

In the theory of factorization, the notion of ``half-factorial domain'' (HFD) may be second only to the unique factorization property with regard to prominence. The standard definition of unique factorization domain describes a domain which is atomic (every nonzero nonunit can be factored into a finite product of irreducible elements) and each nonzero nonunit of $R$ factors uniquely into prime elements. A half-factorial domain is the ``second best'' factorization type from the perspective that all factorizations of an element have the same length. More precisely, we recall that $R$ is an HFD if $R$ is atomic and given two factorizations

\[
\alpha_1\alpha_2\cdots\alpha_n=\beta_1\beta_2\cdots\beta_m
\]

\noindent with each $\alpha_i, \beta_j$ irreducible in $R$, then $n=m$.

 The genesis of the notion of half-factoriality stems from a paper of L. Carlitz (\cite{carlitz}) in which it is shown that if $R$ is the ring of integers in an algebraic number field (the integral closure of $\mathbb{Z}$ in a finite extension of $\mathbb{Q}$), then $R$ is an HFD if and only if the class number of $R$ does not exceed $2$. So in the case of rings of algebraic integers, the class number characterizes HFDs and also neatly divides this collection of HFDs into two nice pieces: if $R$ has class number $1$, then $R$ is a UFD, and if $R$ has class number $2$ then $R$ is an HFD that is not a UFD.
 
 For the half-factorial property, rings of integers have proved to be a useful arena in which the half-factorial property is at the peak of its power and mimics unique factorization in strength. Indeed, if $R$ is a ring of integers that has the HFD property, then this property persists in localizations (see, for example \cite[Corollary 6]{Z}) as well as polynomial (\cite{Co1} or \cite{Z1980}) or power series (\cite{MolesPhD} or \cite{Co2005}) extensions in a finite number of variables, mirroring the behavior of UFDs. One can argue that the lockstep duality of stability properties shared by UFDs and HFDs in the setting of rings of algebraic integers comes from the fact, highlighted first by Carlitz \cite{carlitz}, that both of these classes of domains possess an ideal-theoretic characterization (that is, the class number of $R$ is equal $1$ in the UFD case and is less than or equal to $2$ in the HFD case).
 
 We recall that an order $R$ in a ring of algebraic integers is a subring whose integral closure $\overline{R}$ is equal to the full ring of integers. In the more general setting of orders, we often lose the strong ``integrally closed'' condition reserved to the full ring integers. Nonetheless, the class of orders in rings of integers is still a (comparatively) nice environment to study the half-factorial property; orders in rings of integers enjoy the properties of being Noetherian and $1-$dimensional, and so they have Dedekind integral closure. What is more, although not all ideals in an order are invertible if the order is not integrally closed, all but finitely many prime ideals of such an order are invertible and there is an analog of the class group (see, for example, \cite{hcohn}) that is a useful tool.
 
 It is shown in \cite{Co2} that if $R$ is an order that is an HFD, then $\overline{R}$, the integral closure of $R$, is also an HFD, but the analog of this question is not well understood far outside the confines of rings of integers. In fact, it is shown in \cite{Co4} that this integral closure of an HFD need not be atomic (and hence the integral closure of an HFD need not be an HFD). In \cite{R3} that even if the integral closure of an HFD is still atomic, it still need not be an HFD. Even in the case of orders, if $R$ is an HFD, it is not known if intermediate rings (between $R$ and its integral closure) still have the half-factorial property. We will address this problem in the case of radical conductor in this work as well.
 
 In general ring extensions, many limitations persist. For example, if $R$ is an order in a ring of integers then $R[x]$ is an HFD if and only if $R$ is an integrally closed HFD (see \cite{Z1980} and \cite{Co1}). It is interesting to note that the integrally closed condition is not necessary for the half-factorial property to be preserved in the power series extension. For example, the domain $\mathbb{Z}[\sqrt{-3}]$ is an HFD as is $\mathbb{Z}[\sqrt{-3}][[x]]$, but $\mathbb{Z}[\sqrt{-3}][x]$ is not (\cite[Theorem 6.2]{Co2005}).
 
 Of the classical ``domains of factorization type'' studied by Anderson, Anderson, and Zafrullah in \cite{AAZ}, only the classes of atomic domains and half-factorial domains are unstable under adjunction of a polynomial indeterminate, and this underscores the fact that although HFDs are pretty far up the factorization food chain, it is still an intrinsically unstable class of domains.
 
 These observations give a strong motivation to begin any study of HFD stability by restricting to overrings, and in this setting there is a tool called the boundary map. The boundary map is a function that generalizes the notion of the length function discussed in \cite{AAZ}. Formally, if $R$ is an HFD with quotient field $K$, then the boundary function is a function from the group of nonzero elements of $K$ to the additive group $\mathbb{Z}$ implemented as follows. If $x\in K^*$ we write $x=\frac{a}{b}$, with $a,b\in R\setminus\{0\}$. Further, we factor $a$ and $b$ into irreducibles $a=\pi_1\pi_2\cdots \pi_n$ and $b=\xi_1\xi_2\cdots\xi_m$ and define
 
 \[
 \partial_R(x)=n-m.
 \]

The well-definedness of this assignment is a direct consequence of the fact that $R$ is an HFD.

The boundary map has been used with some success, especially in the arena of orders in rings of integers (see \cite{Co2} for example) and more generally it has been used in investigating overrings of HFDs in which there is some hope of good behavior (for example, (almost) integral extensions). The interested reader should consult \cite{GP2005}, \cite{M2006}, and \cite{M04}, among others.

Almost integral extensions of an HFD are where the boundary map has found the most utility because almost integral elements over an HFD have the property that the boundary is nonnegative 
(\cite{Co2}). But still, the possibility of overrings with nonunits of boundary $0$ is problematic as 
such elements go ``undetected'' by the boundary map. In \cite{Co4}, an example of a nonunit of 
boundary $0$ in an integral extension was given and it was conjectured that this cannot happen if 
the integral extension is atomic; however, in \cite[Example 1.14]{GP2005} this was shown to be false in general. In general, the problem as to when integral atomic extensions can admit a nonunit element of boundary $0$ is still unknown, but in this work, we show that if $R$ is an HFD order in a ring of integers then no integral extension can have a nonunit of boundary $0$.

In the second section of this paper, some key results of general interest are developed. Of these key results, perhaps Proposition \ref{almostintegral} is central in that a general condition is given for the integral closure $T$ of an HFD $R$ to be of the form $T=U(T)R$. It is then shown that in the case of orders of rings of integers, there are no nonunits of boundary $0$. This result gives an affirmative answer to Conjecture I and Conjecture II in \cite{MO2016}. Some of the results in section 2 can be found in with different methods and context in \cite{R25} and for further perspectives on problems of this flavor, the interested reader is encouraged to consult \cite{GP2005}.

In the third section, we apply the results of the second section to show that elasticity is ``squeezed'' in the case of orders of rings of integers with radical conductor. More precisely, it is shown that if $R$ is an HFD order in the ring of algebraic integers $T$, and the conductor ideal $I:=(R:T)$ is radical, then every intermediate ring is also half-factorial.

\section{Overrings of Orders}

\begin{prop}\label{cl2}
Let $D$ be a ring of integers and $R$ an order in $D$. If $R$ is an HFD and $\omega\in D$ is an irreducible such that $\partial_R(\omega)=0$, then $\omega$ is prime in $D$.
\end{prop}

\begin{proof}
We remark at the outset that as $R$ is an order and is an HFD, then $D$ is also an HFD (\cite{Co2}) and hence has $\vert\text{Cl}(D)\vert\leq 2$ (\cite{carlitz}).
Suppose $\omega\in D$ is irreducible and $\partial_R(\omega)=0$. If $\omega$ is not prime, then $(\omega)=PQ$ for nonprincipal primes $P,Q\subset D$ (\cite{carlitz}). We now construct a nonprime irreducible $\pi\in D$ by choosing two nonprincipal prime ideals $A,B$ comaximal to $I:=(R:D)$ and choosing $\pi$ such that $(\pi)=AB$.

We first claim that $\partial_R(\pi)=1$. To see this note that as $\pi$ is comaximal to $I$ and $D/I$ is finite, there exists $k\in\mathbb{N}$ such that $\pi^k+I=1+I$, and hence $\pi^k\in R$ (so $\partial_R(\pi)>0$). As an element of $R$, we can factor 

\[
\pi^k=\xi_1\xi_2\cdots\xi_m
\]

\noindent with each $\xi_i\in R$ irreducible and so $k\partial_R(\pi)=m$. In particular $m\geq k$.

Lifting this factorization back up to $D$ we have the ideal factorization

\[
(\pi)^k=(\xi_1)(\xi_2)\cdots(\xi_m)=A^kB^k.
\]

We now note that each $(\xi_i)$ is principal, and so by uniqueness of prime ideal factorization, each $(\xi_i)$ must factor into an even number of prime ideals (at least two). Hence the number of prime ideal factors of $(\xi_1)(\xi_2)\cdots(\xi_m)$ is at least $2m$ and is exactly equal to $2k$. Hence $2k\geq 2m$ and so $k\geq m$, and by the above, we now have equality. Since $k=m$ we have that $\partial_R(\pi)=1$.

Note that in $D$, we have $PA$ and $QB$ are principal and generated by irreducibles $\alpha$ and $\beta$ respectively. Since (up to units) $\alpha\beta=\omega\pi$, we must have one of $\alpha, \beta$ of boundary $0$ (we will select $\alpha$ without loss of generality). 

Rearranging the ideal factorizations, we have $PB=(\beta')$ and $QA=(\alpha')$. Since (again, up to units) $\alpha'\beta'=\omega\pi$, one of $\alpha'$ or $\beta'$ must be of boundary $0$. Notice that if $\partial_R(\beta')=0$ then $\partial_R(\alpha\beta')=0$, but $(\alpha\beta')=PAPB=P^2(AB)$. Since $P^2$ is principal and $AB$ is generated by an element of boundary $1$, this is a contradiction. On the other hand, if $\partial_R(\alpha')=0$, then up to units $\alpha\alpha'=\omega a$ where $a$ is an irreducible generator for the principal ideal $A^2$. Since $A$ contains $\pi$ and $\pi^k\equiv 1\text{mod}(I)$, $\pi$ is a unit mod $I$ and so $A$ and $I$ must be comaximal. Hence there exists  $j\in \mathbb{N}$ such that $a^j\in R$ (\cite{Co2}), and so $\partial_R(a^j)=j\partial_R(a)>0$ which in turn gives that $\partial_R(a)>0$. From this we observe that $0=\partial_R(\alpha\alpha')=\partial_R(\omega)+\partial_R(a)>0$, and since $\omega$ is irreducible, this contradiction shows that $\omega$ must be prime.
\end{proof}

We now record a useful corollary.

\begin{cor}\label{k2}
Let $R$ be an HFD order in the ring of integers $D$. If $z\in D$ has the property that $\partial_R(z)=0$ then $z$ is contained in no nonprincipal prime ideal.
\end{cor}

\begin{proof}
We suppose that $\partial_R(z)=0$. If $z$ is a unit, then the conclusion is immediate, so we assume that $z=p_1p_2\cdots p_n$ where each $p_i\in D$ is an irreducible with $\partial_R(p_i)=0$. By Proposition \ref{cl2}, each $p_i$ is prime and so if $P\subset D$ is a nonprincipal prime containing $z$, then $p_i\in P$ for some $i$ which, since $D$ is Dedekind and $p_i$ is prime, contradicts the fact that $P$ is nonprincipal.
\end{proof}

\begin{prop}\label{finite0}
If $R$ is an HFD order in a ring of algebraic integers $D$, then the set of irreducibles $\pi\in D$ with the property that $\partial_R(\pi)=0$ is finite (up to associates in $D$).
\end{prop}

\begin{proof}
We first note that if $\pi\in D$ is irreducible with the property that $(\pi)$ is comaximal to $I$ then the coset $\pi+I$ is a unit in $D/I$. Since $D/I$ is finite, there exists $k\in\mathbb{N}$ such that $\pi^k+I=1+I$, and hence $\pi^k\in R$. So in this case, $\pi^k$ is a nonunit element of $R$, and we have that $\partial_R(\pi^k)=k\partial_R(\pi)>0$ which contradicts the assumption that $\partial_R(\pi)=0$. We conclude that any irreducible of boundary $0$ cannot be comaximal to $I$ and since any such irreducible is prime by Proposition \ref{cl2}, it is now immediate that there are only (at most) finitely many irreducible (prime) elements of $D$ that have boundary $0$.
\end{proof}

For the next result, we assume that $R$ is an HFD and $T$ is a half-factorial overring of $R$ that is boundary nonnegative (in the sense that if $x\in T$, then $\partial_R(x)\geq 0$). Note that if $R$ is a half-factorial order in a ring of algebraic integers, then any integral extension is boundary nonnegative, and what is more, the integral closure of $R$ is an HFD, so our assumptions naturally generalize the situation in which $R$ is an order in a ring of algebraic integers.

\begin{prop}\label{almostintegral}
Let $R$ be an HFD and $T$ be a half-factorial, boundary nonnegative overring of $R$. 
\begin{enumerate}
\item If $x\in\text{Irr}(T)$ is prime then $\partial_R(x)$ is equal to either $0$ or $1$.
\item If $x\in T$ is irreducible and $T$ is almost integral over $R$ (in particular, if $T$ is integral over $R$), then $\partial_R(x)$ is equal to either $0$ or $1$.
\item If $T$ is a UFD and $x\in T$ is almost integral and is a product of boundary $1$ primes, then $x$ is associated (by a unit in $T$) to an element of $R$. 
\item If $R$ is an HFD order in the ring of integers $T$, then any product of strictly positive boundary irreducibles is associated by a unit in $T$ to an element of $R$.
\end{enumerate}
\end{prop}

\begin{rem}
It is worth noting here that if $R$ is an HFD, $T$ is an UFD overring, and the extension $R\subseteq T$ is almost integral then (3) implies that each element of $T$ not divisible by a nonunit of boundary $0$ is associated (via a unit in $T$) to an element in $R$. As an application, if $R$ is an order in a ring of integers and its integral closure (the full ring of integers) $T$ is a UFD, then $R$ and $T$ are ``close" in the sense that each element of $T$ that has no boundary $0$ nonunit divisor can be multiplied by a unit to obtain a corresponding element of $R$. We will also see later that in this particular case, there are no boundary $0$ nonunits, and hence we can conclude (with the addition of (4) above) that $T=U(T)R$. This result can also be found in \cite{R25} along with an interesting characterization of HFDs in the setting of rings of algebraic integers.
\end{rem}

\begin{proof}
Suppose first that $x\in T$ is of positive boundary and prime. Since $T$ is an overring of $R$, there exists $0\neq r\in R$ such that $rx\in R$. We first suppose that $\partial_R(r)=m\in\mathbb{N}_0$ and consider $rx\in R$. Note that 

\[
rx=\pi_1\pi_2\cdots\pi_s
\]

\noindent where each $\pi_i$ is an irreducible element of $R$ and $s=m+\partial_R(x)$. We first observe that as $x$ is prime in $T$, we can say that for some $1\leq i\leq s$ that $\pi_i=xt$ for some $t\in T$. Since $\partial_R(\pi_i)=1$ and $T$ is boundary nonnegative, we see that $\partial_R(x)$ must be $0$ or $1$, and this establishes the first statement. 

For the second statement, we suppose that $x$ is an irreducible element of $T$ and $\partial_R(x)\geq 2$. Since $x$ is almost integral over $R$, there is a nonzero $r\in R$ such that $rx^n\in R$ for all $n\in\mathbb{N}$. So for each $n\in\mathbb{N}$, the length of each factorization of $rx^n$ as an element of $R$ is equal to $\partial_R(rx^n)\geq \partial_R(r)+2n$. As an element of $T$, the length of $rx^n$ is equal to $k+n$ where $k$ is the (fixed) length of the factorization $r$ as an element of $T$. So for all $n$ such that $n>k-\partial_R(r)$, we have that $\partial_R(r)+2n>k+n$, which means the length of the factorization of $rx^n$ as an element of $R$ exceeds its length as an element of $T$. Since $T$ is a boundary nonnegative extension of $R$, we have that each irreducible factor in $R$ must remain a nonunit in $T$; as $T$ is an HFD, this gives our desired contradiction.


For the third statement, we first suppose that $x\in T$ is a boundary $1$ prime that is almost integral over $R$ (suppose that, using the notation above, that $rx^n\in R$ for all $n\in\mathbb{N}$). We look more closely at $r$ by factoring it into primes of $T$:

\[
r=z_1z_2\cdots z_kp_1p_2\cdots p_m
\]

\noindent where each $z_i, p_j$ is prime in $T$ with $\partial_R(z_i)=0$ and $\partial_R(p_j)=1$.

We now consider the element $rx^{k+1}\in R\subseteq T$. Since $\partial_R(rx^{k+1})=k+m+1$, then as an element of $R$, we have that 

\[
rx^{k+1}=\xi_1\xi_2\cdots\xi_{k+m+1}
\]

\noindent with each $\xi_i\in\text{Irr}(R)$. But since each $\xi_i$ is of boundary $1$ and $T$ is factorial, each $\xi_i$ is of the form

\[
\xi_i=
\begin{cases}
z_{i_1}z_{i_2}\cdots z_{i_k}p_j,\text{ or}\\
p_j,\text{ or}\\
z_{i_1}z_{i_2}\cdots z_{i_k}x,\text{ or}\\
x
\end{cases}
\]

\noindent up to a unit in $T$. Also precisely $k+1$ of the elements $\xi_i$ must take on the last two forms listed as an element of $T$ (that is there are precisely $k+1$ of the irreducibles that have $x$ as a prime factor in $T$). But no more than $k$ of these can be of the form $z_{i_1}z_{i_2}\cdots z_{i_k}x$ and so for some $1\leq i\leq k+m+1$, $\xi_i$ is associated in $T$ to $x$. It now follows that if $t\in T$ is a product of primes that are all of positive boundary then $t$ is associated to an element of $R$.


For the last statement, we recall that $T$ must be an HFD (\cite{Co2}) and suppose that $x\in T$ is an irreducible with $\partial_R(x)=1$. If $x$ is prime then the proof that $x$ is associated with an element in $R$ is almost identical to the previous case. So we will consider the case where $T$ is an HFD that is not a UFD (and so has class number $2$ by \cite{carlitz}) and $(x)=PQ$ for prime ideals $P,Q\subset T$. We declare that $I=(R:T)$ is the conductor and factor it into prime ideals $I=P_1P_2\cdots P_m$ in $T$ and use the notation $J'=J\bigcap R$ where $J$ is an ideal of $T$.

We declare $Z$ be the set of principal prime ideals of $T$ of boundary $0$; this set is finite by Proposition \ref{finite0}. Now using prime avoidance, we select $\displaystyle\alpha\in (P'Q')\setminus(\bigcup_{k\in S} P'_k\bigcup Z)$ where $S$ is the set of indices in $k\in\{1,2,\ldots,m\}$ such that $P_k$ is distinct from $P$ and $Q$. We now factor $\alpha=\pi_1\pi_2\cdots\pi_t$ with each $\pi_i$ irreducible in $R$. If, as an ideal of $T$, some $(\pi_i)=PQ=(x)$ then we are done. If not, then as $\alpha\in PQ$, we have for some $1\leq i<j\leq t$, $(\pi_i)=PA$ and $(\pi_j)=QB$ for prime ideals $A,B\subset T$ that are comaximal to the conductor $I$.

We now observe that $A'$ and $B'$ are prime ideals of $R$ that are comaximal to $I$ and hence invertible. Additionally, we note that $A'$ (resp. $B'$) is nonprincipal. Indeed, if $A'=(a)$ then as $A$ (resp. $A'$) and $I$ are comaximal, we have $A=A'T=aT$, which is a contradiction. Hence $A'$ (and $B'$) are nonprincipal.

As a porism to the proof of that main result from \cite{carlitz}, $\vert \text{Pic}(R)\vert\leq 2$ and so $A'B'$ is a principal ideal which we will say is generated by an irreducible $\xi\in R$ (and $\xi$ is also irreducible in $T$ as $\xi$ is not contained in any principal prime ideal of boundary $0$). Since $\pi_i\in A'$ and $\pi_j\in B'$, we have that $\pi_i\pi_j\in (\xi)$ and so $\pi_i\pi_j=\xi r$ for some $r\in R$. We now claim that $AB=(\xi)T$. Were this not the case, then $ABJ=(\xi)T$ for some ideal $J\subset T$, but as $AB$ must be principal, this forces $J$ to be principal. Hence if $J$ is a proper ideal, $\xi$ cannot be irreducible. We conclude that $J=T$ and $AB=(\xi)T$ and so $(r)T=PQ=(x)$ and we are done.
\end{proof}

To facilitate the upcoming discussion, we now introduce the notion of partial conductors. In what follows, we will assume that $R$ is an HFD order in its full ring of integers $D$.

\begin{defn}
Let $\omega\in D$ be a boundary $0$ irreducible. We say that the element $d\in D$ partially conducts $\omega$ if there is an element $z\in D$ such that $\partial_R(z)=0$, and $dz\omega\in R$.
\end{defn}

\begin{prop}\label{2}
Suppose $R$ is an HFD order with ring of integers $D$. If $\omega\in D$ is a prime of boundary $0$, then the partial conductors of $\omega$ cannot all lie in the same maximal ideal of $D$.
\end{prop}

\begin{proof}
If this were not the case then there is a prime ideal $P\subset D$ such that if $x\in D$ with $xz\omega\in R$ for some boundary $0$ element $z$, then $x\in P$. Note that as $\omega\in D$ is prime, we have that $\omega D\bigcap R$ is prime, and so by assumption, we have that $\omega D\bigcap R\subseteq P\bigcap R$. As $R$ is $1-$dimensional, and both ideals are nonzero, this containment must be an equality. To see that this gives a contradiction, consider a nonzero element $p\in P\bigcap R$. If equality holds then $p\in\omega D$ and we can write

\[
p=\omega^m d
\]

\noindent with $d\in D$ such that $\omega$ does not divide $d$. But as $d$ is also a partial conductor of $\omega$, this means that $d$ is also in $\omega D$ which is our contradiction, and so not all partial conductors of $\omega$ can all be contained in the same maximal ideal.
\end{proof}

\begin{cor}\label{cor2}
Suppose $R$ is an HFD order with ring of integers $D$. If $\omega\in D$ is a prime of boundary $0$, then $\omega$ is partially conducted to $R$ by infinitely many comaximal boundary $1$ irreducible elements of $D$.
\end{cor}

\begin{proof}
We first remark that if $\omega$ is conducted to $R$ by an element $d\in D$ then $d\omega=\pi_1\pi_2\cdots\pi_k$ where each $\pi_j$ is an irreducible element of $R$. Since (as an element of $D$) $\omega$ must divide some $\pi_i$, we have that $\pi_i=\omega d'$. As $\partial_R(\pi_i)=1$, we conclude that $d'$ can be expressed in the form $d'=zd''$ where $\partial_R(z)=0$ and $d''$ is an irreducible element of $D$. The upshot is that if $\omega$ is our prime in $D$ of boundary $0$, it can be conducted to $R$ by an irreducible in $D$ of boundary $1$ times an element of boundary $0$. This construction, coupled with Proposition \ref{2}, allows us to construct two boundary $1$ partial conductors that are comaximal.


Now suppose that we have comaximal boundary $1$ irreducibles of $D$ $\{p_1, p_2,\ldots,p_n\}, n\geq 2$, that all partially conduct $\omega$ to $R$. For all $1\leq i\leq n$, we write $ p_iz_i\omega\in R$, with each $z_i$ of boundary $0$. We now consider 

\[
\alpha:=p_1z_1\omega+p_2z_2\omega p_3z_3\omega\cdots p_nz_n\omega\in R.
\]

\noindent As $\alpha\in(\omega)\bigcap R$, it is a nonunit and so $\partial_R(\alpha)>0$. If we factor $\alpha$ into irreducibles in $R$, one of them must be of the form $pz\omega$ for some $z\in D$ with $\partial_R(z)=0$ and $p$ an irreducible in $D$ of boundary $1$. Suppose that $p$ and some $p_i$ are not comaximal. In the first case, we assume that $p$ and $p_1$ are not comaximal and declare that $p, p_1\in P$ for some prime ideal $P\subset D$. This means that $p_2z_2\omega p_3z_3\omega\cdots p_nz_n\omega\in P$, but by assumption, no $p_i\in P$ for $i\geq 2$, and we conclude that $z_2\omega z_3\omega\cdots z_n\omega\in P$. But now if $P$ is principal this would mean it is generated by a boundary $0$ prime (and hence could not contain the boundary $1$ irreducibles $p, p_1$), and if $P$ is not principal, this contradicts Corollary \ref{k2}. Hence $\omega$ is partially conducted to $R$ by infinitely many comaximal boundary $1$ irreducible elements of $D$. The second case (in which the second boundary $1$ irreducible $p_i$ has index $i\geq 2$) is almost identical.
\end{proof}

The next result is a central one in this work. We show that if $R$ is an HFD order in the ring of integers $D$ then there is no nonunit in $D$ of boundary $0$.

\begin{thm}\label{bound0}
Let $R$ be an HFD order in the ring of algebraic integers $D$ and $A$ an intermediate ring. If $z\in A$ has the property that $\partial_R(z)=0$, then $z$ is a unit in $A$.
\end{thm}

\begin{proof}
Since $R\subseteq A\subseteq D$ are integral extensions, it suffices to prove the statement for the case $A=D$.

To this end, we suppose $z\in D$ is a prime that has the property that $\partial_R(z)=0$. Let $S=\{p_i\}$ be an infinite family of comaximal elements that partially conduct $z$ to $R$, and we further stipulate that each element is also comaximal to the conductor $(R:D)$. We observe that if $p$ partially conducts $z$ to $R$ then for some boundary $0$ element $z_1\in T$, $pz_1z$ is in $R$. Additionally, if $p$ is comaximal to the conductor, then $p^n\in R$ for some $n\in\mathbb{N}$ (see \cite{Co2} or the proof of Proposition \ref{cl2}) as is $p^nz_1^nz^n$. So $p^n$ is an element of $R$ that partially conducts $z$ to $R$, and so for ease of notation, we will assume that each $p_i\in S$ is an element of $R$ such that $p_iz_iz\in R$, where $\partial_R(z_i)=0$. We will call $z_iz$ a companion of $p_i$.

We now observe that Proposition \ref{cl2} in tandem with Proposition \ref{finite0} gives that the set of elements of $D$ with boundary $0$ form a monoid, $M$, isomorphic to 
\[
U(D)\oplus(\oplus_{i=1}^n\mathbb{N}_0)
\]

\noindent where $n$ is the number of irreducible (prime) elements of $D$ of boundary $0$ and $U(D)$ denotes the units of $D$.

Focusing on $\oplus_{i=1}^n\mathbb{N}_0\subset \oplus_{i=1}^n\mathbb{Z}$, we note that any set of $n+1$ elements of $\oplus_{i=1}^n\mathbb{N}_0$ must form a linearly dependent subset of $\oplus_{i=1}^n\mathbb{Z}$, and so we select the set of companions of $p_1,p_2,\ldots,p_{n+1}\in S$. As these companions form a linearly dependent set, we obtain the relation (written multiplicatively after discarding any terms with $0$ exponent)

\[
u(zz_1)^{a_1}(zz_2)^{a_2}\cdots(zz_t)^{a_t}=(zz'_1)^{a'_1}(zz'_2)^{a'_2}\cdots(zz'_s)^{a'_s}
\]

\noindent where $u\in U(D)$, each $a_i, a'_i>0$. and $s+t\leq n+1$.

Now since $u\in U(D)$ another appeal to either \cite{Co2} or the proof of Proposition \ref{cl2}, we can take the above displayed equation to the power of $m$, where $u^m\in R$ to obtain (up to a unit in $R$)

\[
\alpha:=(zz_1)^{ma_1}(zz_2)^{ma_2}\cdots(zz_t)^{ma_t}=(zz'_1)^{ma'_1}(zz'_2)^{ma'_2}\cdots(zz'_s)^{ma'_s}.
\]

Using this equation, we see that $\alpha$ is conducted to $R$ by both $p:=p_1^{ma_1}p_2^{ma_2}\cdots p_t^{ma_t}$ and $q:={p'}_1^{ma'_1}{p'}_2^{ma'_2}\cdots {p'}_s^{ma'_s}$ and by construction are both themselves in $R$ and are comaximal. Hence, we can find $r_1,r_2\in R$ such that $r_1p+r_2q=1$; multiplying by $\alpha$ gives

\[
r_1p\alpha+r_2q\alpha=\alpha, 
\]

\noindent and as $p\alpha, q\alpha\in R$, we have that $\alpha\in R$ is of boundary $0$ and is a nonunit as it is divisible by $z$ in $D$. This contradiction completes the proof.
\end{proof}

\section{The Squeeze Theorem}

In this section we show that if $R$ is an order in a ring of integers $T$ such that the conductor ideal $I:=(R:T)$ is radical, then every intermediate domain is an HFD. In other words, in this particular case, the elasticity of intermediate domains never increase as you ``approach'' the integral closure.

\begin{thm}
Let $R$ be an HFD order in a ring of algebraic integers and $T$ an integral overring of $R$. The following conditions are equivalent.
\begin{enumerate}
\item $T$ is an HFD.
\item For every $\pi\in\text{Irr}(T), \partial_R(\pi)=1$.
\end{enumerate}
\end{thm}

\begin{proof}
Since $R$ is an order in a ring of algebraic integers, all of its overrings are $1-$dimensional and Noetherian, and hence atomic. This observation, coupled with Theorem \ref{bound0} shows that  \cite[Theorem 2.5]{Co2} applies.
\end{proof}

\begin{thm}\label{bandaid}
Let $R$ be an HFD order in a ring of algebraic integers and $\overline{R}$ its integral closure. The following conditions are equivalent.
\begin{enumerate}
\item For all $T$ with $R\subseteq T\subseteq \overline{R}$, $T$ is an HFD.
\item For all $\alpha\in\overline{R}$, $\alpha$ is irreducible in $R[\alpha]$ if and only if $\partial_R(\alpha)=1$.
\end{enumerate}
\end{thm}

\begin{proof}
For $(1)\Longrightarrow(2)$, we select a nonzero, nonunit $\alpha\in\overline{R}$. If $\partial_R(\alpha)=1$ then Theorem \ref{bound0} shows that $\alpha$ is irreducible in $R[\alpha]$. On the other hand if $\alpha$ is irreducible in $R[\alpha]$, the fact that $R[\alpha]$ is an HFD, coupled with \cite[Theorem 2.5]{Co2} shows that $\partial_R(\alpha)=1$.

For the other direction, we assume that we can find an intermediate extension $R\subseteq T\subseteq \overline{R}$ that is not an HFD. Since $T$ is an overring of a $1-$dimensional Noetherian ring, it is Noetherian and hence atomic; combining with Theorem \ref{bound0}, we see that once again, \cite[Theorem 2.5]{Co2} applies. Hence there is an irreducible $\pi\in T$ with the property that $\partial_R(\pi)>1$. Since $\pi$ is irreducible in $T$, it certainly must be irreducible in $R[\pi]$ and so we see $(2)$ does not hold and this completes the proof.
\end{proof}

Here is a result of immediate as well as independent interest. 

\begin{prop}\label{uic}
Let $R$ be an integral domain with integral closure $\overline{R}$ and suppose that $\overline{R}=RU(\overline{R})$. If $I\neq J$ are ideals of $R$ then $I\bigcap R\neq J\bigcap R$.
\end{prop}

\begin{proof}
Suppose that $\overline{R}=RU(\overline{R})$ and that $I\neq J$ are ideals in $R$. Without loss of generality, we suppose that $x\in J\setminus I$ and find $u\in U(\overline{R})$ such that $ux\in R$. Note that since $ux\notin I$, $I\bigcap R$ and $J\bigcap R$ are distinct.
\end{proof}

\begin{prop}\label{CRT}
Let $R$ be an order with integral closure $\overline{R}$ and conductor $I$. If if $\overline{R}=RU(\overline{R})$ and (as an ideal of ${\overline R}$) $I=P_1^{e_1}P_2^{e_2}\cdots P_k^{e_k}$ then
\[
{\overline R}/I\cong {\overline R}/P_1^{e_1}\times {\overline R}/P_2^{e_2}\times\cdots\times {\overline R}/P_k^{e_k}
\]

\noindent and

\[
R/I\cong R/Q_1\times R/Q_2\times\cdots\times R/Q_k
\]

\noindent with each $Q_i=P_i^{e_i}\bigcap R$. What is more, if $R\subseteq T\subseteq \overline{R}$ and $x\in {\overline R}$ is such that $x\equiv r_i\text{mod}(P_i^{e_i})$ with $r_i\in T$ for all $1\leq i\leq k$ then $x\in T$.
\end{prop}

\begin{proof}
The first isomorphism is merely the Chinese Remainder Theorem. For the second, we observe that Proposition \ref{uic} gives that if $i\neq j$ then $P_i\neq P_j$ from which it follows that the ideals $Q_i$ are all distinct and comaximal. Once again, the Chinese Remainder Theorem applies.

For the last statement, we first consider the case where $T=R$. The coset $x+I\in{\overline R}/I$ can be identified with the element $(r_1+P_1^{e_1},\ldots, r_k+P_k^{e_k})$ of ${\overline R}/I\cong {\overline R}/P_1^{e_1}\times {\overline R}/P_2^{e_2}\times\cdots\times {\overline R}/P_k^{e_k}$. Let $y\in R$ be such that $y+I$ is identified with $(r_1+Q_1,\ldots, r_k+Q_k)$ of $R/I\cong R/Q_1\times R/Q_2\times\cdots\times R/Q_k$. We now suppose that the image of $y+I$ in ${\overline R}/P_1^{e_1}\times {\overline R}/P_2^{e_2}\times\cdots\times {\overline R}/P_k^{e_k}$ is $(y_1+P_1^{e_1},\ldots, y_k+P_k^{e_k})$. Via the canonical embedding $R/I\longrightarrow{\overline R}/I$ and the fact that $Q_i\subseteq P_i^{e_i}$, we obtain that $y_i\equiv r_i\text{mod}P_i^{e_i}$ for all $1\leq i\leq k$. In turn, this gives that $x\equiv y\ \text{mod} I$ and since $I=(R:{\overline R})$ and $y\in R$, we have that $x\in R$ as desired.

In the more general case, we note that as $\overline{R}=RU(\overline{R})$, then it is certainly the case that $\overline{T}=\overline{R}=RU(\overline{R})=TU(\overline{T})$ and so the hypotheses of Proposition \ref{CRT} hold with $J:=(T:\overline{R})$ containing $I$. As an ideal of $\overline{R}$, $J$ factors as $J=P_1^{a_1}P_2^{a_2}\cdots P_k^{a_k}$ with $0\leq a_i\leq e_i$ for all $1\leq i\leq k$. So if $x\equiv r_i\text{mod}(P_i^{e_i})$ with $r_i\in T$ for all $1\leq i\leq k$, then it is certainly true that $x\equiv r_i\text{mod}(P_i^{a_i})$ and hence the previous case shows that $x\in T$.
\end{proof}

\begin{thm}
Let $R$ be an order that is an HFD with integral closure $\overline{R}$. If $(R:\overline{R})$ is a radical ideal and $T$ is an intermediate extension, then $T$ is an HFD.
\end{thm}

\begin{proof}
Suppose that we have $R\subseteq T\subseteq \overline{R}$. As $R$ is an order in a ring of algebraic integers, $T$ must be atomic (see, for example \cite{CGH2023}) and we showed in the previous section that there are no nonunits of $T$ of boundary $0$. So by \cite[Theorem 2.5]{Co2} it suffices to show in $T$ that every irreducible is of boundary $1$ (with respect to $R$).

To this end, we suppose that $I=(\overline{R}:R)$ is a radical ideal; we factor $I=P_1P_2\cdots P_k$ into a product of distinct prime ideals of $\overline{R}$.

By Theorem \ref{bandaid} it suffices to show that if $\partial_R(\alpha)>1$ then $\alpha$ is reducible in $R[\alpha]$. To this end, we take $\alpha=xy$ with $x,y\in\overline{R}$ and $\partial_R(x)$ and $\partial_R(y)$ both strictly positive. Since $\overline{R}=RU(\overline{R})$, we can rearrange units so that $y\in R\subseteq R[\alpha]$.

By Proposition \ref{CRT}, we have that

\[
\overline{R}/I\cong \overline{R}/P_1\times\overline{R}/P_2\times\cdots\times\overline{R}/P_k.
\]

We now consider the image of $x$ in each $\overline{R}/P_i$, and choose an element of $\overline{R}$ using the Chinese Remainder Theorem as follows. For each prime ideal $P_i$ choose \[
x_i=\begin{cases}
x,\text{ if }y\equiv 0\ \text{mod}(P_i), \\
1, \text{\ otherwise}
\end{cases}
\]
 and select an element $x'\in\overline{R}$ that is a solution to this family of congruences. From Proposition \ref{almostintegral} we can find a unit $u\in U(\overline{R})$ such that $ux'\in R$.
 
 We now claim that both $ux$ and $u^{-1}y$ are elements of $R[xy]$. By Proposition \ref{CRT} we merely need to show that for all $1\leq i\leq k$ both $ux$ and $u^{-1}y$ are equivalent to an element of $R[xy]$ $\text{mod} P_i$. 
 
 For $ux$ our claim certainly holds if $x\equiv 0\ \text{mod} P_i$, so we suppose that $x\not\equiv 0\ \text{mod} P_i$. In this case, if $y\equiv 0\text{mod} P_i$ then by the construction of $x'$, we have that $u\equiv x^{-1}r\text{mod} P_i$ for some $r\in R$ and so $ux\equiv r\text{mod} P_i$. On the other hand, if $y\not\equiv 0\text{mod} P_i$ then $u\equiv r\text{mod} P_i$ and $xy\equiv r'\text{mod} P_i$. Since $y,r'$ are equivalent to units $\text{mod} P_i$, we have that $ux\equiv ur'y^{-1}\text{mod} P_i$. We also note that $u, y$ are units in $R$ $\text{mod}(P_i\bigcap R)$ and $r'\in R[xy]\text{mod}(P_i\bigcap R)$. So in all subcases, $ux\in R[xy]\text{mod}(P_i\bigcap R)$ and this verifies the claim for $ux$.
 
 For $u^{-1}y$ we first note that, as above, we have our desired result if $y\equiv 0\text{mod} P_i$. If $y\not\equiv 0\text{mod} P_i$, then we again consider two cases. In the first case, we assume that $x\equiv 0\text{mod} P_i$; in this case, we note that $x'=1$ and so $u$ and hence $u^{-1}$ is equivalent to an element of $R$ and so $u^{-1}y$ is equivalent to an element of $R$. In the second case, we assume that $x\not\equiv 0\text{mod}(P_i)$, which is the analog of the last case in the previous paragraph. In that case, we observed that $y$ is equivalent to a unit $\text{mod}(P_i\bigcap R)$ and hence $u^{-1}y$ is equivalent to an element of $R$ $\text{mod}(P_i\bigcap R)$.

So in all cases, we have that $ux, u^{-1}y$ are equivalent to an element in $R[xy]$ $\text{mod}(P_i\bigcap R)$ for all $i$ and so by Proposition \ref{CRT}, we have that $ux, u^{-1}y\in R[\alpha]$ and hence $\alpha=(ux)(u^{-1}y)$ reduces in $R[\alpha]$. We conclude that any intermediate ring is an HFD.
\end{proof}

\bibliography{biblio2}{}
\bibliographystyle{plainurl}

\end{document}